\documentclass[a4paper,11pt]{amsart}

\usepackage[utf8]{inputenc}
\usepackage[english]{babel}
\usepackage[bitstream-charter]{mathdesign}
\usepackage[T1]{fontenc}
\usepackage{amsfonts}
\usepackage{geometry}
\usepackage{amsmath, amsthm}
\usepackage[colorlinks = true]{hyperref}
\usepackage{color}
\usepackage{pgf,tikz}
\usetikzlibrary{arrows}
\usepackage{cancel}
\usepackage{color}
\usepackage{subfigure}
\usepackage{float}
\usepackage{verbatim}
\usepackage{wrapfig}
\usepackage{mathtools}

\usepackage[normalem]{ulem}

\usepackage{pdfsync}
\usepackage{color}
\usepackage{mathrsfs}

\usepackage[T1]{fontenc}
\usepackage{mathrsfs}
\usepackage{epsfig}
\usepackage{xypic}
\usepackage{verbatim}
\usepackage{multicol}
\usepackage{wrapfig}
\usepackage{pgf,tikz}
\usepackage{mathrsfs}
\usetikzlibrary{arrows}

\numberwithin{equation}{section}

	\usepackage{amsrefs}

\newcommand{\Dom}{\mathrm{Dom}}

\newcommand{\N}{\mathbb{N}}

\newcommand{\R}{\mathbb{R}}

\usepackage{mathrsfs}
\usepackage{mathtools}
\usepackage{comment}
\usepackage{indentfirst}
\usepackage{braket}
\usepackage{scalerel}

{\left\lbrace\begin{array}{@{}l@{}}}%
{\end{array}\right.}

\makeatletter
\newcommand*{\mint}[1]{%
	\mint@l{#1}{}%
}
\newcommand*{\mint@l}[2]{%
	\@ifnextchar\limits{%
		\mint@l{#1}%
	}{%
	\@ifnextchar\nolimits{%
		\mint@l{#1}%
	}{%
	\@ifnextchar\displaylimits{%
		\mint@l{#1}%
	}{%
	\mint@s{#2}{#1}%
}%
}%
}%
}
\newcommand*{\mint@s}[2]{%
	\@ifnextchar_{%
		\mint@sub{#1}{#2}%
	}{%
	\@ifnextchar^{%
		\mint@sup{#1}{#2}%
	}{%
	\mint@{#1}{#2}{}{}%
}%
}%
}
\def\mint@sub#1#2_#3{%
	\@ifnextchar^{%
		\mint@sub@sup{#1}{#2}{#3}%
	}{%
	\mint@{#1}{#2}{#3}{}%
}%
}
\def\mint@sup#1#2^#3{%
	\@ifnextchar_{%
		\mint@sub@sup{#1}{#2}{#3}%
	}{%
	\mint@{#1}{#2}{}{#3}%
}%
}
\def\mint@sub@sup#1#2#3^#4{%
	\mint@{#1}{#2}{#3}{#4}%
}
\def\mint@sup@sub#1#2#3_#4{%
	\mint@{#1}{#2}{#4}{#3}%
}
\newcommand*{\mint@}[4]{%
	\mathop{}%
	\mkern-\thinmuskip
	\mathchoice{%
		\mint@@{#1}{#2}{#3}{#4}%
		\displaystyle\textstyle\scriptstyle
	}{%
	\mint@@{#1}{#2}{#3}{#4}%
	\textstyle\scriptstyle\scriptstyle
}{%
\mint@@{#1}{#2}{#3}{#4}%
\scriptstyle\scriptscriptstyle\scriptscriptstyle
}{%
\mint@@{#1}{#2}{#3}{#4}%
\scriptscriptstyle\scriptscriptstyle\scriptscriptstyle
}%
\mkern-\thinmuskip
\int#1%
\ifx\\#3\\\else_{#3}\fi
\ifx\\#4\\\else^{#4}\fi  
}
\newcommand*{\mint@@}[7]{%
	\begingroup
	\sbox0{$#5\int\m@th$}%
	\sbox2{$#5\int_{}\m@th$}%
	\dimen2=\wd0 %
	\let\mint@limits=#1\relax
	\ifx\mint@limits\relax
	\sbox4{$#5\int_{\kern1sp}^{\kern1sp}\m@th$}%
	\ifdim\wd4>\wd2 %
	\let\mint@limits=\nolimits
	\else
	\let\mint@limits=\limits
	\fi
	\fi
	\ifx\mint@limits\displaylimits
	\ifx#5\displaystyle
	\let\mint@limits=\limits
	\fi
	\fi
	\ifx\mint@limits\limits
	\sbox0{$#7#3\m@th$}%
	\sbox2{$#7#4\m@th$}%
	\ifdim\wd0>\dimen2 %
	\dimen2=\wd0 %
	\fi
	\ifdim\wd2>\dimen2 %
	\dimen2=\wd2 %
	\fi
	\fi
	\rlap{%
		$#5%
		\vcenter{%
			\hbox to\dimen2{%
				\hss
				$#6{#2}\m@th$%
				\hss
			}%
		}%
		$%
	}%
	\endgroup
}

\makeatletter
\def\overbracket#1{\mathop{\vbox{\ialign{##\crcr\noalign{\kern3\p@}
				\downbracketfill\crcr\noalign{\kern3\p@\nointerlineskip}
				$\hfil\displaystyle{#1}\hfil$\crcr}}}\limits}
\def\underbracket#1{\mathop{\vtop{\ialign{##\crcr
				$\hfil\displaystyle{#1}\hfil$\crcr\noalign{\kern3\p@\nointerlineskip}
				\upbracketfill\crcr\noalign{\kern3\p@}}}}\limits}
\def\overparenthesis#1{\mathop{\vbox{\ialign{##\crcr\noalign{\kern3\p@}
				\downparenthfill\crcr\noalign{\kern3\p@\nointerlineskip}
				$\hfil\displaystyle{#1}\hfil$\crcr}}}\limits}
\def\underparenthesis#1{\mathop{\vtop{\ialign{##\crcr
				$\hfil\displaystyle{#1}\hfil$\crcr\noalign{\kern3\p@\nointerlineskip}
				\upparenthfill\crcr\noalign{\kern3\p@}}}}\limits}
\def\downparenthfill{$\m@th\braceld\leaders\vrule\hfill\bracerd$}
\def\upparenthfill{$\m@th\bracelu\leaders\vrule\hfill\braceru$}
\def\upbracketfill{$\m@th\makesm@sh{\llap{\vrule\@height3\p@\@width.7\p@}}%
	\leaders\vrule\@height.7\p@\hfill
	\makesm@sh{\rlap{\vrule\@height3\p@\@width.7\p@}}$}
\def\downbracketfill{$\m@th
	\makesm@sh{\llap{\vrule\@height.7\p@\@depth2.3\p@\@width.7\p@}}%
	\leaders\vrule\@height.7\p@\hfill
	\makesm@sh{\rlap{\vrule\@height.7\p@\@depth2.3\p@\@width.7\p@}}$}
\makeatother

\theoremstyle{theorem}
\newtheorem{theorem}{\sc Theorem}[section]  
\newtheorem{proposition}[theorem]{\sc Proposition}   
        
\newtheorem{lemma}[theorem]{\sc Lemma}

\theoremstyle{remark}
\newtheorem{definition}[theorem]{\sc Definition}

\newtheorem{remark}[theorem]{Remark}

\usepackage{amsfonts}

\newcommand{\dd}{\mathrm{d}}

\begin{document}
\title[Maximal Hypoellipticity on Lie Groups]{Maximal Hypoellipticity for Left-Invariant Differential Operators on Lie Groups}

\author[Tommaso Bruno]{Tommaso Bruno}

\address{Tommaso Bruno \\Dipartimento di Scienze Matematiche ``Giuseppe Luigi Lagrange'', Politecnico di Torino, Corso Duca degli Abruzzi 24, 10129 Torino, Italy}
\email{tommaso.bruno@polito.it}

\maketitle

\renewcommand{\labelitemi}{$\bullet$}
\begin{abstract}
Given a differential operator defined in terms of left-invariant vector fields on a Lie group, we prove that the local condition defining maximal hypoellipticity is equivalent to a global estimate if the operator is left invariant. As a consequence, we characterize the domain of the closure in $L^p$ of certain maximal hypoelliptic differential operators on test functions on Lie groups.
\end{abstract}

\section{Introduction}
The notion of maximal hypoellipticity for a differential operator of degree $m$ was first introduced by Helffer and Nourrigat (cf.~\cite{HelffNourr} or~\cite[Chapter 8]{HelffNier}) and it consists, for the smooth functions supported on a given bounded open set, on a control on the Sobolev norm of order $m$ in terms of the graph norm of the operator, with a constant depending on the chosen set. It seems a natural but unexplored question whether this local condition is equivalent to a global estimate when the operator and the norms involved enjoy some sort of translation invariance. The aim of this paper is to investigate on this problem. 

We consider matrix-valued, differential operators defined in terms of left-invariant vector fields on Lie groups endowed with a left measure. We answer in the affirmative when the operator is left invariant, or more generally when its principal part is left invariant and the lower order terms have essentially bounded coefficients (Theorem~\ref{teocontrolN}). The matrix-valued case is aimed at considering differential operators acting on differential forms of arbitrary degree~(see recent works by Baldi, Franchi and Tesi~\cite{BFT1,BFT2,BFT3}). Indeed, this property allows us to characterize the domain of the closure in $L^p$ for suitable $p$'s of some maximal hypoelliptic differential operators on Lie groups, including Rumin's Laplacians~\cite{Rumin1} on the contact structure of the Heisenberg group (Proposition~\ref{Chardomain}). In addition to its own interest, this result can be combined with a result of Marchi~\cite{Marchi} to obtain the essential self-adjointness on $L^2$ of Rumin's Laplacians defined on test horizontal forms. This was indeed our initial motivation in considering the problems of this paper.

\subsection{Notation}
Let $G$ be a noncompact, connected Lie group and identify its Lie algebra $\mathfrak{g}$ with the algebra of left-invariant vector fields on $G$. For every positive integer $k$ and multi-index $I=(I_1,\dots, I_k)$ of length $ |I|=k$, denote with $X_I$ the differential operator $X_{I_1}\cdots X_{I_k}$, where $X_{I_j}\in \mathfrak{g}$. We shall consider operators $\mathcal{P}\colon C^\infty(G)^\mathfrak{n} \to C^\infty(G)^\mathfrak{m}$ for arbitrary $\mathfrak{n},\mathfrak{m}\in \N$ of the form
\[
\mathcal{P} (\alpha_1,\dots,\alpha_\mathfrak{n})= \left(\sum_{i=1}^\mathfrak{n}\mathcal{P}_{i1}\alpha_i, \, \dots, \, \sum_{i=1}^\mathfrak{n} \mathcal{P}_{i\mathfrak{m}}\alpha_i\right), \qquad \alpha = (\alpha_1, \dots, \alpha_\mathfrak{n} )\in C^\infty(G)^\mathfrak{n},
\]
where $\mathcal{P}_{ij}$ are (scalar-valued) differential operators on $G$ given by
\begin{equation}\label{diffop}
\mathcal{P}_{ij}=\sum_{0\leq |I|\leq m} a^I_{ij}\,X_I
\end{equation}
for some measurable coefficients $(a^I_{ij})$ on $G$, vector fields $X_{I_j}$ in $\mathfrak{g}$, and $m\in \N$. Equivalently, if $ (\eta_1,\dots,\eta_\mathfrak{n})$ is a basis of $\R^\mathfrak{n}$ and $(\xi_1,\dots,\xi_\mathfrak{m})$ is a basis of $\R^\mathfrak{m}$, we may write
\begin{equation}\label{difopN}
\mathcal{P}\alpha = \sum_{j=1}^\mathfrak{m} \sum_{i=1}^\mathfrak{n} (\mathcal{P}_{ij} \alpha_i)\xi_j, \qquad \alpha = \sum_{i=1}^\mathfrak{n} \alpha_i \eta_i \in C^\infty(G)^\mathfrak{n},
\end{equation}
where $(\mathcal{P}_{ij})$'s are as in~\eqref{diffop}. To shorten the notation, we often identify $\mathcal{P}$ with its components $(\mathcal{P}_{ij})$. Following~\cite{BFT3}, we say that $\mathcal{P}$ is left invariant if $\mathcal{P}_{ij}$ is for every $i,j$, that is if $\mathcal{P}_{ij}$ commutes with left translations. For $x\in G$, by left translation by $x$ we mean the operator $\tau_x^*$ given by $(\tau_x^*u)(y)= u(x^{-1} y)$ for every measurable function $u$. It is defined on vector-valued functions by componentwise action. We say that $\mathcal{P}$ has order $m\in \N$, $m\geq 1$, if all $\mathcal{P}_{ij}$ have orders $\leq m$ and at least one of them has order equal to $m$. We shall call \emph{principal part} of the scalar-valued operator $\mathcal{P}_{ij}$ in~\eqref{diffop} the operator
\[(\mathcal{P}_{ij})_m\coloneqq \sum_{|I|=m} a_{ij}^I \,X_I,\]
while the operator \[(\mathcal{P}_{ij})_{<m}\coloneqq \mathcal{P}_{ij}-(\mathcal{P}_{ij})_m\]
will be called its \emph{lower order terms}. We then set \[\mathcal{P}_m\coloneqq ((\mathcal{P}_{ij})_m),\qquad \mathcal{P}_{<m}\coloneqq ((\mathcal{P}_{ij})_{<m}).\]
Let $\lambda$ be a left Haar measure on $G$. For every measurable subset $\Omega\subseteq G$ and $p\in [1,\infty)$, we denote with $L^p(\Omega,\lambda)$ the space of all (equivalence classes of) measurable functions $u$ such that $|u|^p$ is integrable on $\Omega$ with respect to $\lambda$, and if $p=\infty$, we denote with $L^\infty(\Omega,\lambda)$ the space of (equivalence classes of) functions which are $\lambda$-essentially bounded on $\Omega$. We endow these spaces with the usual norms which we denote $\| \cdot \|_{L^p(\Omega, \lambda)}$. For $\mathfrak{n}\in \N$, we denote with $L^{p,\mathfrak{n}}(\Omega,\lambda) = (L^p(\Omega, \lambda))^\mathfrak{n}$ the space of $\mathfrak{n}$-vectors whose entries are in $L^p(\Omega, \lambda)$. We shall endow $L^{p,\mathfrak{n}}(\Omega,\lambda)$ with the norm
\[\|\alpha\|_{L^{p,\mathfrak{n}}(\Omega,\lambda)} = \sum_{i=1}^\mathfrak{n} \|\alpha_i\|_{L^p(\Omega,\lambda)}.\]
If $\Omega= G$, we will write $L^{p,\mathfrak{n}}(\lambda)$. We will often also omit to specify the measure $\lambda$ and write simply $L^p$, $L^p(\Omega)$, $\| \cdot \|_p$, and so on, to avoid cumbersome notation. The space of smooth and compactly supported $\mathfrak{n}$-vectors $(C_c^\infty(\Omega))^\mathfrak{n}$ will be denoted $C_c^{\infty,\mathfrak{n}}(\Omega)$.

Denote with $\mathfrak{X}=\{X_1,\ldots,X_\nu\}$ a family of linearly independent left-invariant vector fields in $\mathfrak{g}$. We shall say that an operator $\mathcal{P}$ as in~\eqref{difopN} is a differential operator \emph{along derivatives of $\mathfrak{X}$} if $X_{I_j}\in \mathfrak{X}$ for every $I$ appearing in~\eqref{diffop}. If $u\in C^\infty_c(G)$, $p\in [1,\infty]$ and $\ell\in \N_+$, we define
\[
|u|^\mathfrak{X}_{\ell,p}=\sum_{|I|=\ell} \|X_I u\|_p, \quad X_{I_j}\in \mathfrak{X}, \quad  j=1,\dots,\ell
\]
and for $\alpha \in C^{\infty,\mathfrak{n}}_c(G)$, we set
\[
|\alpha|^{\mathfrak{X}}_{\ell,p}= \sum_{i=1}^\mathfrak{n} |\alpha_i|^{\mathfrak{X}}_{\ell,p}.
\]
Observe that we do not require a priori that $\mathfrak{X}$ satisfies H\"ormander's condition.  If this is the case, however, and if $p\in (1,\infty)$, then
\[
\sum_{\ell=0}^m |u|^\mathfrak{X}_{\ell,p} \approx \| u\|_{L^p_m(\lambda)}
\]
where $L^p_m(\lambda)$ is one of the Sobolev spaces defined in~\cite{BPTV} (see in particular~\cite[Proposition 3.3]{BPTV}, and also Section~\ref{Secdomain} below). In view of the results therein, the choice of the {left} Haar measure on $G$ seems rather natural.

\subsection{Maximal Hypoellipticity and results}

We are now ready to define maximal hypoellipticity. The reader should compare our definition with that of~\cite{HelffNourr}.
\begin{definition}\label{Def:maxhypo}
Let $\mathcal{P}$ be a differential operator of degree $m$ and let $\Omega \subseteq G$ be open. If $\mathfrak{X}$ is a family of left-invariant vector fields and $p\in [1,\infty]$, we say that $\mathcal{P}$ is $\mathfrak{X}$-maximal hypoelliptic in $L^p(\Omega)$ if there exists a constant $C(\Omega)$ such that
\begin{equation}\label{maxhypo}
\sum_{\ell=0}^m|\alpha|^{\mathfrak{X}}_{\ell,p}\leq C(\Omega)\, (\|\mathcal{P}\alpha\|_p+\|\alpha\|_p) \qquad\forall \alpha\in C^{\infty,\mathfrak{n}}_c(\Omega).
\end{equation}
We say that $\mathcal{P}$ is $(\mathfrak{X},p)$-maximal hypoelliptic if for every $x\in G$ there exists an open neighbourhood $\Omega_x$ of $x$ such that $\mathcal{P}$ is $\mathfrak{X}$-maximal hypoelliptic in $L^p(\Omega_x)$.
\end{definition}
If $G$ is a stratified group and the family $\mathfrak{X}$ is a basis of the first layer of its Lie algebra, for every bounded open $\Omega \subset G$ our definition of $\mathfrak{X}$-maximal hypoellipticity in $L^2(\Omega)$ coincides with the maximal hypoellipticity on $\Omega$ considered in~\cite{BFT3}. Outside stratified groups, examples of $\mathfrak{X}$-maximal hypoelliptic operators in $L^2(\Omega)$ for some $\Omega \subset G$ are Rumin's Laplacians~\cite{Rumin1} on the group $\mathrm{SL}(2,\R)$ and on the groups of~\cite{Khak,Diatta}, when $\mathfrak{X}$ is a basis of the kernel of the contact 1-form of the group (see also~\cite{BaudGar}). The case $p \in (1,\infty)$ was also considered in~\cite{BFT3}.

\smallskip

Observe that if $\Omega_0 \subseteq \Omega$ is open and $\mathcal{P}$ is $\mathfrak{X}$-maximal hypoelliptic in $L^p(\Omega)$, then it is also in $L^p(\Omega_0)$. Moreover, if $\mathcal{P}$ is left invariant and $\mathfrak{X}$-maximal hypoelliptic in $L^p(\Omega)$ for some $\Omega \subseteq G$, then it is $(\mathfrak{X},p)$-maximal hypoelliptic. Indeed, if $x\in G$, one can left translate $\Omega$ to get a neighbourhood $\Omega_x$ of $x$, and $\mathcal{P}$ is $\mathfrak{X}$-maximal hypoelliptic in $L^p(\Omega_x)$ by its left invariance and the left invariance of the norm.

It seems a natural question whether a left-invariant differential operator $\mathcal{P}$ which is $\mathfrak{X}$-maximal hypoelliptic in $L^p(\Omega)$ for some open $\Omega \subset G$ is also $\mathfrak{X}$-maximal hypoelliptic in $L^p(G)$. In other words, whether the local estimate~\eqref{maxhypo} for a left-invariant operator $\mathcal{P}$ implies the global estimate
\[
\sum_{\ell=0}^m|\alpha|^{\mathfrak{X}}_{\ell,p}\leq C \, (\|\mathcal{P}\alpha\|_p+\|\alpha\|_p) \qquad\forall \alpha\in C^{\infty,\mathfrak{n}}_c(G)
\]
for some constant $C$ independent of $\alpha$. The main result of this paper is the following.
\begin{theorem}\label{teocontrolN}
Let $p \in [1, \infty)$, let $\mathfrak{X}$ be a family of left-invariant vector fields and $\mathcal{P}$ be a matrix-valued differential operator of degree $m$ along derivatives of $\mathfrak{X}$. Assume that
\begin{itemize}
\item[\emph{1.}] $\mathcal{P}_m$ is left invariant;
\item[\emph{2.}] $\mathcal{P}_{<m}$ has $L^\infty$ coefficients.
\end{itemize}
Then, if $\mathcal{P}$ is $\mathfrak{X}$-maximal hypoelliptic in $L^p(\Omega)$ for some $\Omega \subseteq G$, it is $\mathfrak{X}$-maximal hypoelliptic in $L^p(G)$.
\end{theorem}
We shall prove Theorem~\ref{teocontrolN} by proving that (i) the statement holds when $\mathcal{P}$ itself is left invariant, and (ii) that maximal hypoellipticity is preserved under $L^\infty$-perturbations of the lower order terms (Proposition~\ref{propperturb}). This is an analogous result to~\cite[Theorem 7]{BramBrand} for local estimates. As already mentioned, as an application of Theorem~\ref{teocontrolN} we characterize the domain of the closure of some maximal hypoelliptic differential operators on test functions on Lie groups, including Rumin's Laplacian on the contact structure of the Heisenberg group (Proposition~\ref{Chardomain}).

We observe that the study of maximal hypoellipticity of matrix-valued operators cannot be reduced to that of scalar-valued operators. Indeed, if all components $\mathcal{P}_{ij}$ of $\mathcal{P}$ are $\mathfrak{X}$-maximal hypoelliptic in $L^p(\Omega)$ for some $\Omega$, then $\mathcal{P}$ is also, but the converse is not true. As an example, if $\Delta$ is the Laplacian on $\R^n$ and $\mathfrak{X}= \{\partial_1,\dots, \partial_n \}$ is the canonical basis of vector fields on $\R^n$, then the operator
\[\mathcal{P}= \begin{pmatrix}
\Delta & 0\\
0 & \Delta
\end{pmatrix} \colon C^{\infty,2}(\R^n) \to C^{\infty,2}(\R^n)\]
is $\mathfrak{X}$-maximal hypoelliptic in $L^2(\R^n)$ since the Laplacian $\Delta$ is, but its off-diagonal components are not. Observe moreover that if $G$ is a symmetric, nonnegative, left-invariant and homogeneous differential operator of even order on a stratified group, and if $\mathfrak{X}$ is a basis of the first layer of $\mathfrak{g}$, then $(\mathfrak{X},2)$-maximal hypoellipticity is equivalent to hypoellipticity (cf.~\cite[Theorem 1.1]{BFT3}). 

\section{From local to global estimates}
We begin by proving a Nash $\varepsilon$-type inequality, which is the object of the following lemma.
\begin{lemma}\label{lemmaNash}
Let $m\geq 2$, $p\in [1,\infty]$ and $\mathfrak{X}$ be a family of $\nu$ left-invariant vector fields. Then there exists a constant $C(\nu,m)$ such that for every $\ell\leq m-1$ and every $\epsilon\in (0,1]$
$$
|\alpha|^{\mathfrak{X}}_{\ell,p}\leq\epsilon^{m-\ell} |\alpha|^{\mathfrak{X}}_{m,p}+C(\nu,m)\epsilon^{-\ell} \| \alpha\|_p \qquad\forall \alpha\in C^{\infty,\mathfrak{n}}_c(G).
$$
\end{lemma}
\begin{proof}
Observe first of all that the vector-valued case is a straightforward consequence of the scalar case. Thus, we can limit ourselves to prove the statement when $\mathfrak{n}=1$, i.e.\ for $\alpha= u\in C_c^\infty(G)$. Moreover, since the family of vector fields $\mathfrak{X}$ plays no role, to simplify the notation we shall write simply $|u|_{\ell,p}$ instead of $|u|_{\ell,p}^{\mathfrak{X}}$.

Let $\epsilon\in (0,1]$. The starting point is the inequality
\begin{equation}\label{eqrobin}
\|X_j u\|_p \leq \varepsilon \|X_j^2 u\|_p + 2\varepsilon^{-1} \|u\|_p
\end{equation}
which can be found in the proof of~\cite[Lemma III.3.3]{Robinson}, by considering the right-translation as a representation of $G$ into $L^p(\lambda)$. Taking the sum over $j=1,\dots,\nu$ we then get
\begin{equation}\label{eq1}
|u|_{1,p}\leq\epsilon |u|_{2,p}+2\nu\epsilon^{-1} \| u\|_p.
\end{equation}
By induction on $\ell$, we now prove that for every $\varepsilon>0$
\begin{equation}\label{eq2}
|u|_{\ell,p}\leq\epsilon |u|_{\ell+1,p}+C(\nu,\ell)\epsilon^{-\ell} \| u\|_p.
\end{equation}
If $\ell=1$ this is~\eqref{eq1}. Hence, assume~\eqref{eq2} holds for $\ell-1\leq m-2$. Then, by~\eqref{eq1}
\begin{align}
|u|_{\ell,p}&=\sum_{|I|=\ell-1} |X_I u|_{1,p} \nonumber\\&
\leq\sum_{|I|=\ell-1} (\epsilon |X_I u|_{2,p}+2\nu\epsilon^{-1} \| X_I u\|_p) \nonumber \\
&= \epsilon |u|_{\ell+1,p}+2\nu\epsilon^{-1} |u|_{\ell-1,p}. \label{eq3}
\end{align}
Since
\[
|u|_{\ell-1,p}\leq\delta |u|_{\ell,p}+C(\nu,\ell-1) \,\delta^{-(\ell-1)} \| u\|_p \qquad\forall \delta>0
\]
by the inductive assumption, if $\delta=\epsilon/(4\nu)$ we get from~\eqref{eq3} 
\[
|u|_{\ell,p}\leq\epsilon\,|u|_{\ell+1,p}+\frac{1}{2}|u|_{\ell,p}+C(\nu,\ell-1)2^{2\ell-1} \nu^{\ell}\, \epsilon^{-\ell}\,\|u\|_p,
\]
and thus
\[
|u|_{\ell,p}\leq2\epsilon\,|u|_{\ell+1,p}+C(\nu,\ell-1)2^{2\ell} \nu^{\ell} \epsilon^{-\ell}\,\|u\|_p \qquad\forall \epsilon>0.
\]
This is~\eqref{eq2} with $2\epsilon$ instead of $\epsilon$. The statement is then proved applying~\eqref{eq2} backwards from $\ell=m-1$.
\end{proof}

\begin{remark}
If $G$ is a homogeneous group and $(\delta_r)$ is a family of dilations on $G$ (see e.g.~\cite[Definition 1.3.24]{Bonfiglioli}) a more direct proof of~\eqref{eqrobin} is the following. By Taylor's formula applied to the function $ t \mapsto u(x\exp(tX_j))$, we get
\[
u(x\exp X_j)=u(x)+X_ju(x)+\int_0^1 (1-t) X_j^2 u(x\exp tX_j) \, \dd t.
\]
Thus, by the invariance of the measure and Minkowski's integral inequality
\[
\|X_j u\|_p\leq\frac{1}{2}\|X_j^2 u\|_p +2\|u\|_p,
\]
and this inequality applied to $u\circ \delta_{2\epsilon}$ gives~\eqref{eqrobin}.
\end{remark}
By Lemma~\ref{lemmaNash} we can now prove that maximal hypoellipticity of an operator $\mathcal{P}$ is preserved under $L^\infty$-perturbations of its lower order terms.
\begin{proposition}\label{propperturb}
Let $\mathfrak{X}$ be a family of left-invariant vector fields, and $\mathcal{P}$ be a differential operator of degree $m$ along derivatives of $\mathfrak{X}$ such that $\mathcal{P}_{<m}$ has $L^\infty(\Omega)$ coefficients for some open $\Omega \subseteq G$. If $p\in [1,\infty]$, then: 
 \begin{itemize}
\item[\emph{1.}] if $\mathcal{P}$ is $\mathfrak{X}$-maximally hypoelliptic in $L^p(\Omega)$, then there exists a constant $C=C(\Omega)$ such that $\|\mathcal{P}\alpha\|_p \leq C(\|\mathcal{P}_m \alpha\|_p + \|\alpha\|_p)$ for every $\alpha\in C_c^{\infty,\mathfrak{n}}(\Omega)$;
\item[\emph{2.}] if $\mathcal{P}_m$ is $\mathfrak{X}$-maximal hypoelliptic in $L^p(\Omega)$, then there exists a constant $C'=C'(\Omega)$ such that $\|\mathcal{P}_m \alpha\|_p \leq C'(\|\mathcal{P} \alpha\|_p + \|\alpha\|_p)$ for every $\alpha\in C_c^{\infty,\mathfrak{n}}(\Omega)$. 
\end{itemize}
In particular, $\mathcal{P}$ is $\mathfrak{X}$-maximal hypoelliptic in $L^p(\Omega)$ if and only if $\mathcal{P}_m$ is.
\end{proposition}
\begin{proof}
Let $A$ be the maximum of the $L^\infty(\Omega)$ norms of $\mathcal{P}_{<m}$'s coefficients, and let $\alpha \in C_c^{\infty,\mathfrak{n}}(\Omega)$. Then, by Lemma~\ref{lemmaNash}, for every $\epsilon\in (0,1]$ 
\begin{align*}
\|\mathcal{P}_{<m}\alpha\|_p
&\leq C(A,\mathfrak{m}) \bigg( \sum_{\ell=1}^{m-1} |\alpha|^{\mathfrak{X}}_{\ell,p}+\|\alpha\|_p\bigg) \\ 
&\leq C'(A,\mathfrak{m}) \bigg( \sum_{\ell=1}^{m-1}\Big(\epsilon^{m-\ell} |\alpha|^{\mathfrak{X}}_{m,p}+C(\nu,m)\epsilon^{-\ell}\|\alpha\|_p\Big)+\|\alpha\|_p\bigg)\\ 
&\leq C(A,\mathfrak{m},\nu,m) \Big(\epsilon|\alpha|^{\mathfrak{X}}_{m,p}+\epsilon^{-m+1}\|\alpha\|_p\Big).
\end{align*}
Therefore, if we choose $\epsilon=\big(2C(A,\mathfrak{m},\nu,m)\ C(\Omega)\big)^{-1}$ we get
\[
\|\mathcal{P}_{<m}\alpha\|_p\leq\frac{1}{2C(\Omega)} |\alpha|^{\mathfrak{X}}_{m,p}+C(\nu,m,A,\Omega)\|\alpha\|_p.
\]
Suppose now that $\mathcal{P}$ is $\mathfrak{X}$-maximal hypoelliptic in $L^p(\Omega)$. Then
\begin{align*}
\|\mathcal{P}\alpha\|_p&\leq\|\mathcal{P}_m\alpha\|_p+\|\mathcal{P}_{<m}\alpha\|_p \\
 &\leq\|\mathcal{P}_m\alpha\|_p+ \frac{1}{2C(\Omega)}|\alpha|^{\mathfrak{X}}_{m,p}+C(\nu,m,A,\Omega)\|\alpha\|_p\\ 
&\leq\|\mathcal{P}_m\alpha\|_p+ \frac{1}{2}\big(\| \mathcal{P}\alpha\|_p+\|\alpha\|_p\big)+C(\nu,m,A,\Omega)\|\alpha\|_p,
\end{align*}
from which $\|\mathcal{P}\alpha\|_p\leq C (\|\mathcal{P}_m \alpha\|_p+\|\alpha\|_p)$ follows.
\par
The point $2.$ is proved in the same way.
\end{proof}
The last ingredient to prove Theorem~\ref{teocontrolN} is a covering lemma, which is an easy generalisation of a theorem of Anker. Its proof is a straightforward adaptation of~\cite[Lemma 1]{Anker}, and is omitted.
\begin{lemma}\label{Anker}
Let $G$ be a locally compact second-countable group, and let $U$, $V$ be two open relatively compact subsets of $G$ such that $U\subset V$. Then there exists a countable family $\mathfrak{U}=\{x_k\colon k\in \N\}\subset G$ and $n\in \N$ such that
\begin{itemize}
\item[\emph{1.}] $G=\bigcup_{k} x_k U$;
\item[\emph{2.}] $\forall x_0 \in \mathfrak{U}$, $x_0 V \cap x V$ is non-empty for at most $n$ elements $x\in \mathfrak{U}$.
\end{itemize} 
\end{lemma}
Property 2 of Lemma~\ref{Anker} will be referred to as \emph{bounded overlap property}. We are now ready to prove Theorem~\ref{teocontrolN}.
\begin{proof}[of Theorem~\ref{teocontrolN}]
For notational convenience, for every open $\Omega\subseteq G$ and $0\leq \ell \leq m$ we define
\[
\{u\}^\mathfrak{X}_{\ell,p}=\sum_{|I|=\ell} \|X_I u\|_p^p, \quad u \in C_c^\infty(\Omega),\qquad  \{\alpha\}^{\mathfrak{X}}_{\ell,p}= \sum_{i=1}^\mathfrak{n} \{\alpha_i\}^{\mathfrak{X}}_{\ell,p}, \quad  \alpha \in C^{\infty,\mathfrak{n}}_c(\Omega).
\]
All the vector fields are meant to belong to $\mathfrak{X}$. Under this new notation, the $\mathfrak{X}$-maximal hypoellipticity of $\mathcal{P}$ in $L^p(\Omega)$ is equivalent to saying that
\begin{equation}\label{maxhypov2}
\sum_{\ell=0}^m\{\alpha\}^{\mathfrak{X}}_{\ell,p}\leq C(\Omega)\, (\|\mathcal{P}\alpha\|_p^p+\|\alpha\|_p^p) \qquad\forall \alpha\in C^{\infty,\mathfrak{n}}_c(\Omega).
\end{equation}
Moreover, from Lemma~\ref{lemmaNash} we infer the existence of a constant $\kappa=\kappa(p,\nu,m)$ such that for every $\ell\leq m-1$ and every $\epsilon>0$
\begin{equation}\label{Lemmakappav2}
\{\alpha\}^{\mathfrak{X}}_{\ell,p}\leq\kappa (\epsilon^{p(m-\ell)} \{\alpha\}^{\mathfrak{X}}_{m,p}+\epsilon^{-p\ell} \|\alpha\|_p^p) \qquad\forall \alpha\in C^{\infty,\mathfrak{n}}_c(G).
\end{equation}
We divide the proof into two steps. 

\smallskip

\emph{Step 1}. We assume that $\mathcal{P}$ is left invariant, and prove the statement. 

Since $\mathcal{P}$ is left invariant, it is $(\mathfrak{X},p)$-maximal hypoelliptic, hence we may assume that $\Omega$ is an open neighbourhood of the identity $e\in G$. Let $U$, $V$ be open sets such that $U\subset V \subset \Omega$. By Lemma~\ref{Anker}, there exists a countable family $(x_k)$ of elements of $G$ such that $G= \bigcup_k x_k U$ and each $x_k V$ intersects at most $n$ other $x_k V$. We set $U_k\coloneqq x_k U$ and $V_k \coloneqq x_k V$.

Let now $\psi \in C_c^\infty(V)$ such that $\psi \equiv 1$ on $U$, $\psi \equiv 0$ on $V^c$ and $\psi$ takes values in $[0,1]$; set $\psi_k = \tau_{x_k}^*\psi$. Let $\alpha\in C_c^{\infty,\mathfrak{n}}(G)$. Then
\[\sum_{\ell=0}^m \{\alpha\}^\mathfrak{X}_{\ell,p}\leq \sum_k \sum_{\ell=0}^m \{(\alpha\psi_k)\}^\mathfrak{X}_{\ell,p} = \sum_k \sum_{\ell=0}^m \{\tau_{x_k^{-1}}^*(\alpha\psi_k)\}^\mathfrak{X}_{\ell,p} \]
since $\psi_k \equiv 1$ on $U_k$, $(U_k)$ is a covering of $G$, and both the $L^p$ norm and the vector fields in $\mathfrak{X}$ are left invariant. 

Observe that $\tau_{x_k^{-1}}^*(\alpha\psi_k)$ is supported in $V$. Thus, by the $\mathfrak{X}$-maximal hypoellipticity of $\mathcal{P}$ in $L^p(V)$,
\begin{align}\label{Palphapsik}
\sum_{\ell=0}^m \{\alpha\}^\mathfrak{X}_{\ell,p} & \leq C(V) \sum_k \Big[ \|\tau_{x_k^{-1}}^*(\alpha\psi_k)\|_{p}^p+ \|\mathcal{P}(\tau_{x_k^{-1}}^*(\alpha\psi_k))\|^p_{p}\Big] \nonumber \\&\leq C(V,n) \Big[\|\alpha\|_p^p +\sum_k \|\mathcal{P}(\alpha\psi_k)\|_{p}^p \Big]
\end{align}
where we used the bounded overlap property of the family $(V_k)$, and the left invariance of $\mathcal{P}$ and of the $L^p$-norm. 

By Leibniz's formula,
\[\mathcal{P}(\alpha\psi_k) = \psi_k \mathcal{P}\alpha + \mathcal{R}(\psi_k; \alpha)\]
where $\mathcal{R}(\psi_k;\alpha)$ is an $\mathfrak{m}$-vector supported in $V_k$ whose entries are sum of terms having at least one derivative of $\psi_k$, at most $m-1$ derivatives of the components of $\alpha$, and all the derivatives are along the vector fields in $\mathfrak{X}$. By the bounded overlap property and~\eqref{Palphapsik}
\begin{equation*}
\begin{split}
\sum_{\ell=0}^m \{\alpha\}^\mathfrak{X}_{\ell,p} & 
\leq C(V,n) \Big[ \|\mathcal{P}\alpha\|_p^p + \|\alpha\|_p^p + \sum_{k} \| \mathcal{R}(\psi_k;\alpha)\|_{p}^p \Big]. 
\end{split}
\end{equation*}
It remains to control the last term. If we denote $\max_{|I|\leq m}\, \|X_I \psi\|_\infty$ by $E$, then by~\eqref{Lemmakappav2} we obtain that for every $\epsilon\in (0,1)$
\begin{align*}
\sum_{k} \| \mathcal{R}(\psi_k;\alpha)\|_{p}^p
&\leq C(\mathfrak{m},m,E)\, n \sum_{\ell=0}^{m-1} \{ \alpha\}^\mathfrak{X}_{\ell,p} \\
&\leq C(\mathfrak{m},m,E,n,\kappa)\, \Big(\epsilon^p \{\alpha\}^\mathfrak{X}_{m,p}+C \epsilon^{-p(m-1)} \|\alpha\|_{p}^p\Big).
\end{align*}
We used again the bounded overlap property of the sets $(V_k)$ and left invariance of the vector fields in $\mathfrak{X}$. Thus, by choosing $\epsilon^p=(2C(\mathfrak{m},m,E,n,\kappa)C(V,n))^{-1}$, we obtain the global estimate as desired.

\smallskip

{\emph{Step 2}}. We now consider $\mathcal{P}$ as in the statement.
 
Since $\mathcal{P}$ is $\mathfrak{X}$-maximal hypoelliptic in $L^p(\Omega)$, then $\mathcal{P}_m$ is also by Proposition~\ref{propperturb}. By Step 1 applied to $\mathcal{P}_m$ and by Proposition~\ref{propperturb} again, for $\alpha\in C_c^{\infty,\mathfrak{n}}(G)$ we get
\begin{equation*}
\sum_{\ell=0}^m |\alpha|_{\ell,p}^{\mathfrak{X}} \leq C ( \|\mathcal{P}_m\alpha\|_p + \|\alpha\|_p) \leq \tilde{C} ( \|\mathcal{P}\alpha\|_p + \|\alpha\|_p)
\end{equation*}
which concludes the proof.
\end{proof}

\section{The domain of certain maximal hypoelliptic operators}\label{Secdomain}

We can finally characterize the domain of the closure in $L^{p,\mathfrak{n}}(\lambda)$, $p\in [1,\infty)$, of differential operators which are $\mathfrak{X}$-maximal hypoelliptic in $L^p(\Omega)$ for some $\Omega \subseteq G$, considered on test $\mathfrak{n}$-vectors. When $p=2$, notable examples of such operators are Rumin's Laplacians on the contact structure of the Heisenberg group, see e.g.~\cite{Rumin1,Rumin2,BFT1,BFT2}.

We first introduce Sobolev spaces of integer regularity on the Lie group $G$ (see also~\cite{BPTV}). For every positive integer $\sigma$, we define the Sobolev space
\[
L^{p}_\sigma(\lambda,\mathfrak{X})\coloneqq \{u\in L^p(\lambda)\colon X_I u\in L^p(\lambda), \; |I|\leq \sigma\}
\]
endowed with the norm
\[
\|u\|_{L^p_\sigma(\lambda, \mathfrak{X})} \coloneqq \sum_{0\leq |I|\leq \sigma} \|X_I u\|_{L^p(\lambda)}
\]
where all the vector fields that appear belong to the family $\mathfrak{X}$. The space $L^{p,\mathfrak{n}}_\sigma(\lambda, \mathfrak{X})$ is defined as the space of $\mathfrak{n}$-vector whose entries are in $L^p_\sigma(\lambda, \mathfrak{X})$, with the norm
\[
\|\alpha\|_{L^{p,\mathfrak{n}}_\sigma(\lambda, \mathfrak{X})} \coloneqq \sum_{i=1}^\mathfrak{n} \|\alpha_i\|_{L^p_\sigma(\lambda, \mathfrak{X})}.
\]
\begin{proposition}\label{Chardomain}
Let $\mathfrak{X}$ be a family of left-invariant vector fields. Let $\mathcal{P}$ be a matrix-valued differential operator of degree $m$ along derivatives of $\mathfrak{X}$, whose principal part is left invariant and with lower order terms with $L^\infty$ coefficients. Let $\Dom(\mathcal{P})\coloneqq C_c^{\infty,\mathfrak{n}}(G)$. Assume that $\mathcal{P}$ is $\mathfrak{X}$-maximal hypoelliptic in $L^p(\Omega)$ for some $\Omega \subseteq G$, and closable in $L^{p,\mathfrak{n}}(\lambda)$ for some $p\in [1,\infty)$. If $\overline{\mathcal{P}}$ stands for its closure, then
\[
\Dom(\overline{\mathcal{P}})= \overline{C_c^{\infty,\mathfrak{n}}(G)}^{\; L^{p,\mathfrak{n}}_m(\lambda, \mathfrak{X})}.
\]
\end{proposition}

\begin{proof}
Let $\alpha \in C_c^{\infty,\mathfrak{n}}(G)$. By Theorem~\ref{teocontrolN}, $\mathcal{P}$ is $\mathfrak{X}$-maximal hypoelliptic in $L^p(G)$. Therefore, for every $\alpha \in C_c^{\infty,\mathfrak{n}}(G)$
\[
\|\alpha\|_{L^{p,\mathfrak{n}}_{m}(\lambda, \mathfrak{X})} = \sum_{|I|\leq m}\sum_{i=1}^\mathfrak{n} \|X_I \alpha_i\|_{p} \lesssim (\|\mathcal{P} \alpha\|_p + \|\alpha\|_p) \lesssim \sum_{|I|\leq m}\sum_{i=1}^\mathfrak{n} \|X_I \alpha_i\|_2 =\|\alpha\|_{L^{p,\mathfrak{n}}_{m}(\lambda, \mathfrak{X})}.
\]
In other words, the norm of $L^{p,\mathfrak{n}}_{m}(\lambda, \mathfrak{X})$ and the graph norm of $\mathcal{P}$ are equivalent on $C_c^{\infty,\mathfrak{n}}(G)$. 
\end{proof}

We finally remark that by Proposition~\ref{Chardomain}, a result of Marchi~\cite[Theorem I.3.29]{Marchi} and the density of test functions on the Sobolev spaces on the Heisenberg group (cf.~\cite{Folland}), the $L^2$ closure of Rumin's Laplacians on smooth and compactly supported horizontal differential forms of suitable degree on the contact structure of the Heisenberg group are self-adjoint. In other words, they are essentially self-adjoint. Since the introduction of Rumin's complex goes beyond the scope of this paper, we refer to one of~\cite{Rumin2,BFT1,BFT2,Marchi} and leave the few details to the interested reader.

\smallskip

\subsection*{Acknowledgements} 
I would like to thank Giancarlo Mauceri for several helpful discussions and his constant help and support.

\end{document}